\documentclass[12pt]{amsart}%
\usepackage{amsmath}
\usepackage{amsfonts}
\usepackage{amssymb}
\usepackage{graphicx}%
\providecommand{\U}[1]{\protect\rule{.1in}{.1in}}
\newtheorem{theorem}{Theorem}

\newtheorem{corollary}[theorem]{Corollary}

\newtheorem{example}[theorem]{Example}

\newtheorem{lemma}[theorem]{Lemma}

\newtheorem{proposition}[theorem]{Proposition}

\begin{document}
\title[Fr\'{e}chet completions]{Fr\'{e}chet completions of moderate growth old and (somewhat) new results }
\author{Nolan R. Wallach}
\maketitle

\begin{abstract}
This article has two objectives. The first is to give a guide to the proof of
the (so-called) Casselman-Wallach theorem as it appears in Real Reductive
Groups II. The emphasis will be on one aspect of the original proof that leads
to the new result in this paper which is the second objective. We show how a
theorem of van der Noort combined with a clarification of the original
argument in my book lead to a theorem with parameters (an alternative is one
announced by Berstein and Kr\"{o}tz). This result gives a new proof of the
meromorphic continulation of the smooth Eisenstein series.

\end{abstract}
\dedicatory{{\small Dedicated to David Vogan in honor of his 60th birthday. His deep
contributions to the algebraic aspects of representation theory helped me to
realize that I am an analyst.}}

\section{Introduction.}

In the past three years there has been new activity related to what has come
to be called the Casselman-Wallach Theorem (CW Theorem) notably in the thesis
of Vincent van der Noort [vdN] and the paper of Bernstein and Kr\"{o}tz [BK].
These works are concerned with the problem of inserting dependence on
parameters in our theory of Fr\'{e}chet completions of moderate growth. In
[vdN] there is a development of the theory of admissible $(\mathfrak{g}%
,K)$--modules depending on parameters with a goal of \ finding criteria that
imply that these families are imbeddable in families of smooth Fr\'{e}chet
completions. In [BK] another stated goal is to give a more elementary proof of
the CW Theorem. The complications in \textquotedblleft the\textquotedblright%
\ original proof in part, stem from the fact that there is no joint proof by
Casselman and me of the theorem. Casselman's proof (appearing in [C]) and the
proof of [BK] apply only to linear groups.  My version, which can only be
found in my second volume on real reductive groups ([RRGII]), proves the
theorem in the context of the class of reductive groups that appear in my book
with Borel ([BW]) .It had been my hope to include in this paper a proof of the
theorem for general semi-simple groups with a finite number of connected
components (i.e. allow infinite center for the identity component) but time
constraints made it impossible for me to complete the foundational work
necessary. The version in this paper implies the theorem for the case of
connected semi-simple Lie groups with finite center (e.g. the metaplectic
group or the metalinear group). Except for the greater generality of my
version the proof in [C] and that in [RRGII] are essentially the same (and is
basically the same as that of [BK], except for the handling of spherical
representations). In hindsight I must admit that Casselman's article is more
carefully written than my development (however, his version also has typos).
Beyond its obvious defects in exposition my argument can be deemed complicated
or at least \textquotedblleft non-elementary\textquotedblright\ because it
relies on results that are necessary in other contexts in the study of real
reductive groups. These results include, the asymptotic expansions of
generalized matrix coefficients of admissible representations, the theory of
intertwining operators developed by Vogen and me ([VW]) and the Langlands
quotient theorem. In each case the results are easily stated but their use in
a proof can certainly allow one to say that a proof that doesn't use these
results is more elementary (however [BK] uses the Langlands quotient theorem).
It is definitely not true that more elementary implies easier to understand
(compare Hadamard's proof of the prime number theorem to Selberg's
\textquotedblleft elementary proof\textquotedblright).

This paper is, for the most part, an attempt to explain my proof of the
theorem with some simplifications and (we hope) clarifications. It emphasizes
the places where it differs from that in [C] which can be seen most clearly in
the construction of the minimal completion. This result is given in here
detail with several simplifications. Also in the appendix to the construction,
there are a few results on $C^{\infty}$ vectors that I couldn't find in the
literature. The last section of this paper will show that the beautiful theory
of of van der Noort combined with an extended version of the argument in
subsection 4.3 allows us to add dependence on parameters. In particular we
show that the $C^{\infty}$ Eisenstein series has a meromorphic continuation
from the tube in which it is initially defined. Another reason for an extended
explanation of a 20 year old result is that the unique reference for the
general result is in a book that is out of print.

\section{A litany of notation}

We purpose of this section is to set up the standard notation. Let
$G_{\mathbb{C}}$ be a reductive algebraic group defined over $\mathbb{R}$ and
let $G_{\mathbb{R}}$ be its group of real points. Then a real reductive group
in the sense of [BW] is a finite covering group of an open subgroup of
$G_{\mathbb{R}}$. This can be described in elementary terms (as in [RRGI]) as
follows: Let $G_{\mathbb{R}}$ be a subgroup of $GL(n,\mathbb{R})$ that is the
locus of zeros of polynomial functions on $M_{n}(\mathbb{R})$ and is invariant
under matrix transpose. Then $G$ is a finite covering of an open subgroup of
$G_{o}\subset G_{\mathbb{R}}$ (for a proof that these notions are the same see
e.g. [W]).

It is convenient to assume that $G$ is given by the more elementary
definition. We make that assumption. So $K_{\mathbb{R}}=G_{\mathbb{R}}\cap
O(n)$ is a maximal compact subgroup. If $p:G\rightarrow G_{o}$ is the covering
homomorphism ($G_{o}\subset G_{\mathbb{R}}$ open) and if $K_{o}=K_{\mathbb{R}%
}\cap G_{o}$ then $K=p^{-1}(K_{\mathbb{R}}\cap G_{o})$ is a maximal compact
subgroup of $G$. We also fix an Iwasawa decomposition $G_{\mathbb{R}%
}=K_{\mathbb{R}}A_{\mathbb{R}}N_{\mathbb{R}}$ ($A_{\mathbb{R}}$ a maximal
abelian subgroup consisting of positive definite self adjoint elements and
$N_{\mathbb{R}}$ a maximal unipotent subgroup) so $G_{o}=K_{o}A_{\mathbb{R}%
}N_{\mathbb{R}}$ is an Iwasawa decomposition of $G_{o}$ and if $A$ and $N$ are
the identity components of $p^{-1}(A_{\mathbb{R}})$ and $p^{-1}(N_{\mathbb{R}%
})$ respectively then $KAN$ is an Iwasawa decomposition of $G$. Let $M$ denote
the centralizer of $A$ in $K$ then $MAN$ is a minimal parabolic subgroup of
$G$. As usual, we will call a parabolic subgroup of $G$ containing $MAN$ a
standard parabolic subgroup.

We will write $\left\Vert g\right\Vert $ for the Hilbert-Schmidt norm of%
\[
r(g)=\left[
\begin{array}
[c]{cc}%
p(g) & 0\\
0 & p(g^{-1})^{T}%
\end{array}
\right]  \in GL(2n,\mathbb{R)}%
\]
for $g\in G$. Then $\left\Vert ...\right\Vert $ is a right and left
$K$--invariant norm on $G$ (in the sense of [RRGI]). Recall that a norm is a
function on $G$, $g\mapsto\left\Vert g\right\Vert $, such that if
$S\subset\mathbb{R}$ is compact then ~$\{g\in G|\left\Vert g\right\Vert \in
S\}$ is compact and $\left\Vert xy\right\Vert \leq\left\Vert x\right\Vert
\left\Vert y\right\Vert ,x,y\in G$. We also recall that if $\left\Vert
...\right\Vert $ is a norm on $G$ and if $(\pi,H)$ is a strongly continuous
Banach representation of $G$ then there exist constants $C$ and $r$ such that
$\left\Vert g\right\Vert _{H}\leq C\left\Vert g\right\Vert ^{r}$ (see [RRGI] 2.A.2.2)

We use the standard notation $\mathfrak{g}$ for the complexified Lie algebra
of $G$ (thought of as left invariant complex vector fields on $G$),
$U(\mathfrak{g)}$ for the universal enveloping algebra (left invariant
differential operators) and $Z(\mathfrak{g})$ for the center of
$U(\mathfrak{g})$ (thought of as bi--invariant differential operators). We
define the form $B(X,Y)=trXY$ on $Lie(G)=Lie(G_{\mathbb{R}})$. Then $B$ is
non-degenerate and negative definite on $Lie(K)$. We set $C$ (resp. $C_{K}$)
equal to the Casimir operator on $G$ (resp. $K$) corresponding to $B$.

As usual, a $(\mathfrak{g},K)$--module is a module $V$ for $\mathfrak{g}$ that
is also a $K$ module with the following compatibility properties

1. Set $W_{v}=Span_{\mathbb{C}}(Kv)$ then $\dim W_{v}<\infty$ the action of
$K$ on $W_{v}$ is smooth and the action of $Lie(K)\subset\mathfrak{g}$ on
$W_{v}$ is the same as its action as a real subalgebra of $\mathfrak{g}$.

2. if $k\in K,X\in\mathfrak{g}$ and $v\in V$ then $kXv=\left(  Ad(k)X\right)
kv.$

We say that a $(\mathfrak{g},K)$--module is finitely generated if it is
finitely generated as a $U(\mathfrak{g})$--module.

We denote by $\mathcal{H}(\mathfrak{g},K)$ the category of $(\mathfrak{g}%
,K)$--modules that are admissible (if $W$ is a finite dimensional $K$--module
then $\dim Hom_{K}(W,V)<\infty$ with morphisms the $\mathfrak{g}$ and $K$
module homomorphisms) and finitely generated.

If $V$ is a $(\mathfrak{g},K)$--module and if $\gamma$ is an equivalence class
of irreducible representations of $K$ (that is, an element of $\widehat{K}$)
then we chose a representative, $V_{\gamma}$, for this class and denote by
$V(\gamma)=\sum TV_{\gamma}$ the sum over all $T\in Hom_{K}(V_{\gamma},V)$.
That is, $V(\gamma)$ is the $\gamma$--isotypic component of $V$.

A smooth Fr\'{e}chet module, $(\pi,V)$ for \thinspace$G$ is a homomorphism,
$\pi$, of $G$ to the continuous invertible operators on a Fr\'{e}chet space
$V$ such that the map%
\[
G\times V\rightarrow V
\]
given by%
\[
g,v\mapsto\pi(g)v
\]
is of class $C^{\infty}$ in $G$ and jointly continuous. If we have a smooth
Fr\'{e}chet module $(\pi,V)$ then we can differentiate to get a $\mathfrak{g}%
$--module structure on $V$. We use the notation $V_{K}$ for the space of all
$v\in V$ such that $\dim Span_{\mathbb{C}}(\pi(K)v)<\infty$. Then $V_{K}$ is a
$(\mathfrak{g},K)$ module. If $V_{K}$ is admissible or finitely generated then
we say that $V$ is admissible or finitely generated.

If $V$ is a smooth Fr\'{e}chet module then we say that $V$ is of moderate
growth if for every $\lambda$, a continuous seminorm on $V$, there exists
$\mu$ a continuous semi-norm on $V$ and $r$ real such that
\[
\lambda(\pi(g)v)\leq\left\Vert g\right\Vert ^{r}\mu(v),g\in G,v\in V.
\]

We denote by $\mathcal{HF}_{\operatorname{mod}}(G)$ the category of all smooth
admissible finitely generated Fr\'{e}chet $G$--modules of moderate growth with
morphisms continuous intertwiners with images that are direct summands (in the
category of Fr\'{e}chet spaces).

We now recall the definition of a parabolically induced representation. We
will give details since we will be using the construction in the last section
of this article. In the first sections we will only need the concept for $P$
and for induction from finite dimensional representations. Let $Q$ be a
standard parabolic subgroup of $G$ and let $(\sigma,W)$ be a Hilbert
representation of $Q$ on the Hilbert space $(W,(...,...))$ that is unitary as
a representation of $K\cap Q$. We define $Ind_{Q\cap K}^{K}(W)$ to be the
space of functions $f:K\rightarrow W$ such that $f(mk)=\sigma(m)f(k),m\in
K\cap Q$, $k\in K$ and such that $k\mapsto\left\Vert f(k)\right\Vert $ is in
$L^{2}(K)$. We set for $f,h\in Ind_{K\cap Q}^{K}(W)$%
\[
\left\langle f,h\right\rangle =\int_{K}(f(k),h(k))dk.
\]
thus yielding a Hilbert space structure on $Ind_{K\cap Q}^{K}(W)$. We define
$\pi(k)$ for $k\in K$ to be the right regular action that is $\pi
(k)f(x)=f(xk)$, $x,$ $k\in K$ this yields a unitary representation of $K$. If
$f\in Ind_{K\cap Q}^{K}(W)$ and $q,q^{\prime}\in Q,k,k^{\prime}\in K$ then it
is easily seen that if $qk=q^{\prime}k^{\prime}$ then $\sigma(q)f(k)=\sigma
(q^{\prime})f(k^{\prime})$. Since $G=QK$ we can define $f_{\sigma}%
(g)=\sigma(q)f(k)$ for $qk=g$ and $q\in Q$ and $k\in K$. We define the Hilbert
representation $Ind_{Q\cap K}^{G}(\sigma)$ to have underlying Hilbert space
$Ind_{K\cap Q}^{K}(W)$ and the action of $G$ given by ($\pi_{\sigma
}(g)f)(x)=f_{\sigma}(xg)$ for $g\in G$ and $x\in K$. This defines a Hilbert
representation of $G$, $Ind_{Q}^{G}(\sigma)$. If the representation, $W$, of
$Q$ is finite dimensional then the space of $C^{\infty}$ vectors of
$Ind_{Q}^{G}(\sigma)$ with respect to $G$ are the $C^{\infty}$ functions in
$Ind_{K\cap Q}^{K}(W)$ with the topology defined by the seminorms,
$p_{l}(f)=\left\Vert (I+C_{K})^{l}f\right\Vert $. Let $Q=M_{Q}A_{Q}N_{Q}$ be a
standard Langlands decomposition of $Q$ (c.f. [RRGI]). If $(\sigma,W)$ is an
admissible representation of $M_{Q},\nu\in(Lie(A_{Q}))_{\mathbb{C}}^{\ast}$
then we set $\sigma_{\nu}(man)=a^{\nu+\rho_{Q}}\sigma(m)$ for $m\in M_{Q},a\in
A_{Q},n\in N_{Q}$ here $a^{\mu}=e^{\mu(\log(a))}$ and $\rho_{Q}(h)=\frac{1}%
{2}\mathrm{tr}(adh)_{|Lie\not (N_{Q})}$. We will use the notation
$I_{Q,\sigma,\nu}$ for $Ind_{Q}^{G}(\sigma_{\nu})$, $I_{Q,\sigma,\nu}^{\infty
}$ for $Ind_{Q}^{G}(\sigma_{\nu})^{\infty}$ and $V_{Q,\sigma,\nu}$ for
$Ind_{Q}^{G}(\sigma_{\nu})_{K}^{\infty}.$

If $V$ is a topological vector space over $\mathbb{C}$ then $V^{\prime}$
denotes the continuous elements of , $V^{\ast}$, the full dual space of $V$.

\section{The theorem and some applications}

The theorem referenced as the Casselman--Wallach theorem is

\begin{theorem}
The functor $V\rightarrow V_{K}$ from $\mathcal{HF}_{\operatorname{mod}}(G)$
to $\mathcal{H}(\mathfrak{g},K)$ is an equivalence of categories.
\end{theorem}

In this section we will give some implications. We will sketch the proof in
[RRGII] in the next section. First, since there is an inverse functor we see
that if $V$ is an admissible finitely generated $(\mathfrak{g},K)$ module then
there exists a smooth Fr\'{e}chet module of moderate growth, $W$, such that
$W_{K}$ is isomorphic with $V$. This existence is a direct implication of
Casselman's theorem since Banach representations have moderate growth and
Casselman has shown that there is a Hilbert representation of $G$ $(\pi,H)$
such that the smooth $K$--finite vectors of $H$ yield a $(\mathfrak{g}%
,K)$--module isomorphic with $V$. However, our theorem says something much
stronger: $W$ is \emph{unique} up to isomorphism in the category
$\mathcal{HF}_{\operatorname{mod}}(G)$.

The above says that if $V\in\mathcal{H}(\mathfrak{g},K)$ and $(\pi,H)$ is a
Banach representation of $G$ such that $\left(  H^{\infty}\right)  _{K}$ is
isomorphic with $V$ then although there are many different representations in
the Banach category with this property there is only one object up to
isomorphism in $\mathcal{HF}_{\operatorname{mod}}(G)$ whose $K$--finite
vectors yield a module isomorphic with $V$.

The most important implication is just the expansion of what an equivalence of
categories means for the two categories in the theorem. Let $V,W$ be
admissible finitely generated $(\mathfrak{g},K)$ modules and let
$T:V\rightarrow W$ be a morphism. Let $\overline{V}$ and $\overline{W}$ be any
elements of $\mathcal{HF}_{\operatorname{mod}}(G)$ such that $\overline{V}%
_{K}$ and $\overline{W}_{K}$are respectively isomorphic with $V$ and $W$ then
the induced map $L:\overline{V}_{K}\rightarrow\overline{W}_{K}$ extends to a
continuous $G$ intertwining operator from $\overline{V}$ to $\overline{W}$
with closed image that is a direct summand of $\overline{W}$. Thus the theorem
turns analysis into algebra. We will now give several applications. We denote
by $\overline{V}$ an element of $\mathcal{HF}_{\operatorname{mod}}(G)$ with
$\overline{V}_{K}=V$. We will call $\overline{V}$ a completion of $V$ in
$\mathcal{HF}_{\operatorname{mod}}(G)$ and note that it is unique up to isomorphism.

\subsection{\label{automorphic}Continuous functionals and automorphic forms}

In this subsection we give an application to automorphic forms. We first
describe an implication of the proof in [RRGII] which characterizes
$\overline{V}_{|V}^{\prime}$ for $V\in\mathcal{H}(\mathfrak{g},K)$.

Define $\mathcal{A}_{\operatorname{mod}}(G)$ to be the space of all
$C^{\infty}$ functions $f:G\rightarrow\mathbb{C}$ satisfying the following
three conditions\smallskip

1. There exists $d>0$ and for each $x\in U(\mathfrak{g}),$ $C_{x}>0$ such that
$|xf(g)|\leq C_{x}\left\Vert g\right\Vert ^{d}$.

2. $\dim Z(\mathfrak{g})f<\infty.$

3. $f$ is right $K$--finite.

\begin{theorem}
\label{characterization}Let $V\in\mathcal{H}(\mathfrak{g},K)$ and $\lambda\in
V^{\ast}=Hom_{\mathbb{C}}(V,\mathbb{C})$ and let $\overline{V}$ be its
completion in $\mathcal{HF}_{\operatorname{mod}}(G)$ then $\lambda$ extends to
a continuous functional on $\overline{V}$ if and only if for each $v\in V$
there exists $f_{\lambda,v}\in\mathcal{A}_{\operatorname{mod}}(G)$ such that
$xf_{\lambda,v}(k)=\lambda(kxv)$ for all $k\in K$ and $x\in U(\mathfrak{g}).$
\end{theorem}

We will come back to this theorem in the next section.

We recall that $f\in C^{\infty}(G)$ is an automorphic form on $G$ with respect
to a discrete subgroup of finite covolume, $\Gamma$, if $f\in$ $\mathcal{A}%
_{\operatorname{mod}}(G)$ and $f(\gamma g)=f(g)$ for $\gamma\in\Gamma$. We
denote the space of automorphic forms on $\Gamma\backslash G$ by
$\mathcal{A}_{\operatorname{mod}}(\Gamma\backslash G)$.

We now apply the theorem to automorphic forms. Let for $f\in\mathcal{A}%
_{\operatorname{mod}}(\Gamma\backslash G)$ the space $W$ be the span of the
right translates of $f$ under $K$ and let $V=U(\mathfrak{g})W$. Then
$V\in\mathcal{H}(\mathfrak{g},K)$. Let $\delta(h)=h(e)$ ($e$ the identity
element of $G$). We note that if $x\in U(\mathfrak{g})$ and $k\in K$ then
$\delta(kxf)=xf(k).$  Theorem implies 2 that $\delta$ extends to $\overline
{V}$. This proves that $f(g)=\lambda(gv)$ for $v\in V$ and%
\[
\lambda\in(\overline{V}^{\prime})^{\Gamma}=\{\mu\in\overline{V}^{\prime}%
|\mu\circ\gamma=\mu,\gamma\in\Gamma\}\text{.}%
\]

For example, this implies that the analytically continued Eisenstein series
initially shown to exist for $K$--finite elements of an induced
representations from cuspidal parabolic subgroups extend to the $C^{\infty}$
vectors of these representations.  At the end of the paper we will look at the
implications to the meromorphic continuation of $C^{\infty}$--Eisenstein series.

\subsection{$C^{\infty}$ Helgason conjecture and related results}

Let $G$ be connected with compact center and let $P=MAN$ be a standard minimal
parabolic subgroup with given Langlands decomposition. Here $M\subset K$ is
compact, $A$ is the identity component of a maximal split torus over
$\mathbb{R}$ and $N$ is the unipotent radical. We say that an irreducible
representation of $K$, $(\tau,V_{\tau})$, is small if $\tau_{|M}$ is
irreducible. Clearly a one dimensional representation is small (e.g. the
trivial representation). There are more interesting representations that are
small. For example the spin representation of $Spin(2n+1)$ (thought of as $K$
for for the two--fold cover of $SL(2n+1,\mathbb{R})$) or either half spin
representation of $Spin(2n)$ (thought of as $K$ for the two fold cover of
$SL(2n,\mathbb{R})$).

We first describe a general Poisson \textquotedblleft integral
representation\textquotedblright\ for elements of $\mathcal{A}%
_{\operatorname{mod}}(G)$. If $(\sigma,H)$ is a finite dimensional
representation of $P$ we set $I_{P,\sigma}^{\infty}$ equal to the
representation $C^{\infty}$--induced of $G$ from $\sigma$ on $P$. That is the
Fr\'{e}chet space of $C^{\infty}$ vectors of $Ind_{P}^{G}(\sigma)$ as in
section 2. 

\begin{theorem}
Let $f\in\mathcal{A}_{\operatorname{mod}}(G)$ then there exists a finite
dimensional representation, $\sigma$, of $P$, a continuous linear functional
$\lambda\in\left(  I_{P,\sigma}^{\infty}\right)  ^{\prime}$ and $h\in\left(
I_{P,\sigma}^{\infty}\right)  _{K}$ such that $f(g)=\lambda(\pi_{P,\sigma
}(g)h)=f(g)$ for all $g\in G$.
\end{theorem}

For a proof see [RRGII,Theorem 11.9.2]

Now suppose that $(\tau,H_{\tau})$ is a small $K$--type. Let $\sigma=\tau
_{|M}$. Let $I_{P,\sigma.\nu}^{\infty}\ $and $V_{P,\sigma.\nu}\ $be as in
Section 2. Then Frobenius reciprocity implies that $\dim Hom_{K}(H_{\tau
},I_{P,\sigma.\nu}^{\infty})=1$. This implies there is a homomorphism
$\gamma_{\tau,\nu}:U(\mathfrak{g})^{K}\rightarrow\mathbb{C}$ such that
$\pi_{P.\sigma,\nu}(x)f=\gamma_{\tau,\nu}(x)f$ for all $f$ in the $\tau$
isotypic component of $I_{P,\sigma.\nu}^{\infty}$and all $x\in U(\mathfrak{g}%
)^{K}$. One checks easily that $\gamma_{\tau,\nu}(x)$ is polynomial in $\nu$.
Furthermore since $\sigma=\tau_{|M}$ we see that if $k\in N_{K}(M)$
(normalizer of $M$ in $K$) the representation $\sigma^{k}(m)=\sigma
(k^{-1}mk)=\tau(k)^{-1}\sigma(m)\tau(k)$ this implies that $\gamma_{\tau,\nu
}=\gamma_{\tau,s\nu}$ for $s\in W(A)=N_{K}(A)/M$. The analogue in this context
of the $C^{\infty}$--Helgason conjecture is

\begin{theorem}
Let $f\in\mathcal{A}_{\operatorname{mod}}(G)(\tau)$ be such that
$xf=\gamma_{\tau,\nu}(x)f$ for $x\in U(\mathfrak{g})^{K}$ then if
$V_{P,\sigma.\nu}=U(\mathfrak{g})V_{P,\sigma.\nu}(\tau)$ there exists
$\lambda\in(I_{P,\sigma.\nu}^{\infty})^{\prime}$ and $u\in I_{P,\sigma.\nu
}(\tau)$ such that $f(g)=\lambda(\pi_{P.\sigma,\nu}(g)u)$ for $g\in G$.
\end{theorem}

We note that if $\tau$ is the trivial representation then Kostant [K] has
shown that the condition of the theorem is satisfied for $\operatorname{Re}%
\nu$ in the closed positive Weyl chamber. So Theorem 4 implies the
distribution version of the Helgason Conjecture due to [OS]. We note that
using the fact that the analytic vectors of $I_{P,\sigma.\nu}^{\infty}$ are
the real analytic elements of $I_{P,\sigma.\nu}$ we see that the
distributional version of the conjecture implies the hyperfunction version
that is the same assertion (the original Helgason conjecture) but with $f\in
C^{\infty}(G)$ satisfying only the $K$--condition and $xf=\gamma_{\tau,\nu
}(x)f$ for $x\in U(\mathfrak{g})^{K}$ and $\lambda$ in the continuous dual of
the analytically induced representation.

In his work on split groups over $\mathbb{R}$, Seung Lee [Le] has shown that
if $\tau$ is small, $\dim V_{\tau}>1$ and $G$ is split over $\mathbb{R}$ and
simply laced then the analogue of Kostant's result is true and so the above
theorem completely describes $\mathcal{A}_{\operatorname{mod}}(G)(\tau)$, In
the non-simply laced case there are some exceptions see [Le] for a complete discussion.

\section{The main aspects of our proof}

We divide the proof into the construction of a left exact functor (maximal
globalization) and a right exact functor (minimal globalization) from
$\mathcal{H}(\mathfrak{g},K)$ to $\mathcal{H}\mathcal{F}_{\operatorname{mod}%
}(G)$ and then prove that the two functors are equivalent. The theorems
involved in the proof of the left exact functor will only be sketched. This is
the most complicated part of our proof. It is the argument proving the
existence of the right exact functor that will be given in detail since it
plays a role in the extension of the theorem in the last section of this paper
and also since the original version in [RRGII] was rather convoluted and had
some misprints in unfortunate places.

\subsection{\label{leftsection}Step one: A left exact functor}

The first step is to construct what will be an inverse functor to the
$K$--finite functor. Let $P=MAN$ be a minimal parabolic subgroup of $G$ (as in
section 2). Let $\mathfrak{n}=Lie(N)$ and let $V\in\mathcal{H}(\mathfrak{g}%
,K)$. We first sketch the proof of

\begin{theorem}
\label{left}Let $V\in\mathcal{H}(\mathfrak{g},K)$. Then there exists an object
$\overline{V}\in\mathcal{HF}_{\operatorname{mod}}(G)$ such that

1. $\overline{V}_{K}$ is isomorphic with $V$.

2. If $W\in\mathcal{H}(\mathfrak{g},K)$ and $T:W\rightarrow V$ is a
$(\mathfrak{g},K)$--homomorphism and if $X\in\mathcal{H}\mathcal{F}%
_{\operatorname{mod}}(G)$ is such that $X_{K}$ is isomorphic with $W$ then
there exists a $\mathcal{HF}_{\operatorname{mod}}(G)$--morphism
$S:X\rightarrow\overline{V}$ in such that $S(X_{K})=T(W)$.
\end{theorem}

Using standard methods, one can reduce the proof to the groups $G$ that are
connected and simple with finite center. It is for these groups we will sketch
the proof . We consider the simply connected (complex) Lie group with Lie
algebra $\mathfrak{g}$, $G_{\mathbb{C}}$ and the connected subgroup,
\ $G_{\mathbb{R}}$, corresponding to $Lie(G)$. Then we may assume that $G$ is
a finite covering group of $G_{\mathbb{R}}$. Let $Z$ denote that kernel of the
covering homomorphism. By our assumption $Z$ is a finite abelian group. The
first step in the argument is to prove (see [RRGII,11.8.2)

\begin{theorem}
If $\chi\in\widehat{Z}$ then there exists $\tau$ an irreducible representation
of $K$ such that $\tau_{|M}$ is irreducible and $\tau_{|Z}=\chi I$ .
\end{theorem}

A complete classification of representations of $K$ whose restriction to $M$
is irreducible was given in Seung Lee's thesis [Le]. If $(\sigma,V)$ is an
finite dimensional representation of $P$ then we denote by $Ind_{P}^{G}%
(\sigma)^{\infty}$ the smoothly induced representation from $P$ to $G$ that
is, the $C^{\infty}$vectors of $Ind_{P}^{G}(\sigma)$ (defined in Section
2)\ and by $Ind_{P}^{G}(\sigma)_{K}^{\infty}$ the $K$--finite induced
representation of $P$ to $G$. If $\mu$ is a finite dimensional unitary
representation of $M$ and if $\nu\in\mathfrak{a}_{\mathbb{C}}^{\ast}$ then
$I_{P,\mu,\nu}^{\infty}\ $and $V_{P,\sigma,\nu}$ are as in section 2.

An application of the above theorem is (see [RRGII],Lemma 11.4.5).

\begin{lemma}
Let $(\sigma,V)$ be an irreducible finite dimensional representation of $P$
such that $\sigma_{|Z}=\chi I$. Let $\tau$ be a small representation of $K$
with $\tau(z)=\chi(z)I$ and set $\mu=\tau_{|M}$. Then there exist finite
dimensional $G$--representations $F_{1},...,F_{r}$ and $\nu_{1},...,\nu_{r}%
\in\mathfrak{a}_{\mathbb{C}}^{\ast}$ such that $Ind_{P}^{G}(\sigma)_{K}$ is
equivalent to a quotient of $\oplus V_{P,\mu,\nu_{i}}\otimes F_{i}$.
\end{lemma}

Using Casselman's imbedding theorem and the Artin-Rees Lemma we see that if
$V\in\mathcal{H}(\mathfrak{g},K)$ then $\dim V/\mathfrak{n}^{k}V<\infty$ and
$\cap_{k}\mathfrak{n}^{k}V=\{0\}.$ We note that we have natural morphisms
\[%
\begin{array}
[c]{ccc}
& T_{k} & \\
V & \rightarrow & Ind_{P}^{G}(V/\mathfrak{n}^{k}V)_{K}^{\infty}%
\end{array}
\]
and an that there exists $k_{o}$ such that $T_{k}$ is injective for all $k\geq
k_{o}$. We will take $k_{o}$ to be the minimal choice. Using this and the
theory of intertwining operators we prove: If $V\in\mathcal{H}(\mathfrak{g}%
,K)$ and if $W$ is any element of $\mathcal{F}_{\operatorname{mod}}(G)$ such
that $W_{K}$ is isomorphic with $V$ then the map $T_{k}$ extends to a
homomorphism of $W$ into $Ind_{P}^{K}(V/\mathfrak{n}^{k}V)^{\infty}.$ Using
this we take $\overline{V}$ to be the closure of $T_{k_{o}}(V)$ in
$Ind_{P}^{K}(V/\mathfrak{n}^{k}V)^{\infty}.$ We then use the theory of
asymptotic expansions of elements of $\mathcal{A}_{\operatorname{mod}}(G)$
restricted to $MA$, the above lemma and the result in [VW] to prove what we
call the automatic continuity theorem [RRGII, Theorem 11.4.1]. This result is
complicated and we will content ourself to just giving the reference to it and
the following implication.

\begin{theorem}
Let $V\in\mathcal{HF}_{\operatorname{mod}}(G)$. If $F$ is a finite dimensional
representation of $P$ and if $T:V_{K}\rightarrow Ind_{P}^{G}(F)_{K}^{\infty}$
is a morphism in $\mathcal{H}(\mathfrak{g},K)$ then $T$ extends to a morphism
$V\rightarrow$ $Ind_{P}^{G}(F)^{\infty}$.
\end{theorem}

This result implies

\begin{corollary}
\label{needed}If $V,W\in\mathcal{H}(\mathfrak{g},K)$ and if $T:W\rightarrow V$
is an injective morphism then the closure of $W$ in $\overline{V}$ is
isomorphic with $\overline{W}$ in $\mathcal{HF}_{\operatorname{mod}}(G)$.
\end{corollary}

This easily implies 2. in Theorem \ref{left}

This completes the sketch of the existence of the maximal completion in
$\mathcal{F}_{\operatorname{mod}}(G)$.

We note that for any $V\in\mathcal{H}(\mathfrak{g},K)$, there exists a Hilbert
completion of $\overline{V}$, $H$ such that the $K$--$C^{\infty}$ vectors of
$H$ are the same as the $G$--$C^{\infty}$ vectors and $H^{\infty}=\overline
{V}$. \ This implies that $V\rightarrow$ $\overline{V}$ defines a left exact
faithful functor from $\mathcal{H}(\mathfrak{g},K)$ to $\mathcal{HF}%
_{\operatorname{mod}}(G)$.

\subsection{Step two: A right exact functor}

The next step is the dual assertion: the existence of a minimal completion.

\begin{theorem}
\label{minimal}Let $V\in\mathcal{H(}\mathfrak{g},K)$ then there exists an
object $\overline{\overline{V}}\in\mathcal{HF}_{\operatorname{mod}}(G)$ such
that $\overline{\overline{V}}_{K}$ is isomorphic with $V$ and if
$W\in\mathcal{HF}_{\operatorname{mod}}(G)$ and $T:W_{K}\rightarrow
\overline{\overline{V}}_{K}$ is a surjective morphism then $T$ extends to
continuous surjection of $W$ onto $\overline{\overline{V}}$.
\end{theorem}

We note that there is up to isomorphism only one object in $\mathcal{HF}%
_{\operatorname{mod}}(G)$ with the property above enjoyed by $\overline
{\overline{V}}.$ Also $V\rightarrow\overline{\overline{V}}$ is a right exact
functor from $\mathcal{H(}\mathfrak{g},K)$ to the category $\mathcal{HF}%
_{\operatorname{mod}}(G)$. The CW Theorem will therefore be proved if we can
show that $\overline{V}$ and $\overline{\overline{V}}$ are isomorphic in
$\mathcal{HF}_{\operatorname{mod}}(G)$. If this condition is satisfied for $V$
then we say that $V$ is good (obviously, we want to prove that all admissible
finitely generated $(\mathfrak{g},K)$ modules are good). We now devote a
subsection to the proof of this theorem.

\subsection{\label{proofmin}The proof of Theorem \ref{minimal}}

If $V\in\mathcal{H}(\mathfrak{g},K)$ we denote by $\widehat{V}$ its conjugate
dual module. That is, $\widehat{V}\in\mathcal{H}(\mathfrak{g},K)$ consists of
all of the real linear functionals, $\lambda$, on $V$ satisfying\medskip

1. $\lambda(zv)=\bar{z}\lambda(v)$ for $z\in\mathbb{C}$.

2. Set $k\lambda=\lambda\circ k^{-1}$ then $\dim Span_{\mathbb{C}}%
K\lambda<\infty$.

The action of $\mathfrak{g}$ on $\widehat{V}$ is given by $X\lambda
=-\lambda\circ X$ for $\lambda\in\widehat{V}$, $X\in\mathfrak{g}$.

$\widehat{V}$ is clearly a vector space over $\mathbb{C}$. If $(\pi,H)$ is a
Hilbert representation of $G$ then we define the conjugate dual representation
to $(\pi,H)$ to be $(\hat{\pi},H)$ with $\hat{\pi}(g)=\pi(g^{-1})^{\ast}$.
Obviously, if $\pi$ is unitary $\hat{\pi}=\pi$. If $(\pi,H)$ is a Hilbert
representation of $G$ then we denote by $(\pi,H^{\infty})$ the $G$%
--$C^{\infty}$ vectors and by $(\pi,H^{\infty_{K}})$ the $K-C^{\infty}$
vectors. We denote the $C^{\infty}$--vectors of $G$ relative to $\widehat{\pi
}$ by $\widehat{H}{}^{\infty}$.

We first relate Theorem \ref{characterization} to Theorem \ref{minimal}. We
note that if $V\in$ $\mathcal{HF}_{\operatorname{mod}}(G)$ and if $\lambda\in
V^{\prime}$ and $v\in V_{K}$ then $f(g)=\lambda(gv)$ defines an element of
$\mathcal{A}_{\operatorname{mod}}(G)$. If $V\in$ $\mathcal{H(}\mathfrak{g},K)$
then we denote by $V_{\operatorname{mod}}^{\ast}$ the space of all $\lambda\in
V^{\ast}$ such that for each $v\in V$ there exists $f_{\lambda,v}\in$
$\mathcal{A}_{\operatorname{mod}}(G)$ such that $\lambda(kxv)=xf_{\lambda
,v}(k)$ for $k\in K$ and $x\in U(\mathfrak{g)}$(i.e. $\lambda$ satisfies the
condition in Theorem \ref{characterization}). Thus $V_{\operatorname{mod}%
}^{\ast}$ contains $Z_{|V}^{\prime}$ for any object of $\mathcal{HF}%
_{\operatorname{mod}}(G)$ such that $Z_{K}=V$. We start our argument with

\begin{lemma}
Let $V\in\mathcal{H(}\mathfrak{g},K)$ then there exists a Hilbert
representation of $G$, $(\pi,H)$, such that $H_{K}^{\infty}$ is isomorphic
with $V$, $(\pi,H^{\infty})=(\pi,H^{\infty_{K}})$, $(\widehat{\pi}%
,\widehat{H}{}^{\infty})=(\widehat{\pi},\widehat{H}{}^{\infty_{K}})$ and
$(\widehat{\pi},\widehat{H}{}^{\infty})$ is isomorphic with $\overline
{(\widehat{V})}$ .
\end{lemma}

Note that we may assume that $\pi_{|K}$ is unitary thus $H{}^{\infty_{k}%
}=\widehat{H}{}^{\infty_{K}}$.

\begin{proof}
We note that according to the discussion at the end of subsection
\ref{leftsection} (and in the notation thereof) $\overline{(\widehat{V})}$
\ is isomorphic with the space of $C^{\infty}$ vectors in the closure of
$T_{k}(\widehat{V})$ in the induced Hilbert representation $Ind_{P}%
^{G}(\widehat{V}/\mathfrak{n}^{k}\widehat{V})$ for an appropriate $k$. The
$C^{\infty}$ vectors in this representation are the $K$--$C^{\infty}$ vectors
(this is true for any parabolically induced representation from a finite
dimensional Hilbert representation). \ We note that taking $C^{\infty}$
vectors defines an exact functor from the category of Hilbert representations
to the category of smooth Fr\'{e}chet representations. Thus the $C^{\infty}%
$--vectors in the closure of $T_{k}(\widehat{V})$ in $Ind_{P}^{G}%
(\widehat{V}/\mathfrak{n}^{k}\widehat{V})$ are the $K$--$C^{\infty}$ vectors.
We take $H$ to be the closure of $T_{k}(\widehat{V})$. We now observe that the
conjugate dual representation of $Ind_{P}^{G}(\widehat{V}/\mathfrak{n}%
^{k}\widehat{V})$ is the induced representation $Ind_{P}^{G}(\widehat{\left(
\widehat{V}/\mathfrak{n}^{k}\widehat{V}\right)  }\otimes\delta_{P})$with
$\delta_{P}$ the modular function of $P$ and with underlying Hilbert space
$Ind_{M}^{K}(\widehat{V}/\mathfrak{n}^{k}\widehat{V})$ and the $G$%
--$C^{\infty}$ vectors are the same as the $K$--$C^{\infty}$. Let
$\left\langle ...|...\right\rangle $ denote the conjugate linear
$G$--invariant pairing between $Ind_{P}^{G}(\widehat{V}/\mathfrak{n}%
^{k}\widehat{V})$ and $Ind_{P}^{G}(\widehat{\left(  \widehat{V}/\mathfrak{n}%
^{k}\widehat{V}\right)  }\otimes\delta_{P})$ and set $Z=H^{\perp}$relative to
$\left\langle ...|...\right\rangle $. Then as a Hilbert space the conjugate
dual of $H$ is $Ind_{P}^{G}(\widehat{\left(  \widehat{V}/\mathfrak{n}%
^{k}\widehat{V}\right)  }\otimes\delta_{P})/Z$ . If $(\pi,H)$ is the
corresponding Hilbert representation of $G$ then the underlying object of
\ $\mathcal{H}(\mathfrak{g},K)$ is isomorphic with $V$. The assertions of the
lemma are now clear.
\end{proof}

The crux of the matter is the following theorem which will be proved after we
give two implications.

\begin{theorem}
\label{crux}Let $V\in\mathcal{H(}\mathfrak{g},K)$ and let $(\pi,H)$ be a
Hilbert representation such that $V$ is isomorphic with $H_{K}^{\infty}$ .
Assume that $(\hat{\pi},\widehat{H}{}^{\infty})$ is as an object in
$\mathcal{HF}_{\operatorname{mod}}(G)$ isomorphic with $\overline
{(\widehat{V})}$ then denoting $(\pi,H^{\infty})$ by $Z$ we have
$Z_{|V}^{\prime}=V_{\operatorname{mod}}^{\ast}$ (see the definition preceding
Lemma 11).
\end{theorem}

\begin{corollary}
If $X,Y\in\mathcal{HF}_{\operatorname{mod}}(G)$ are such that $X_{K}$ is
isomorphic with $Y_{K}$ and $X_{|X_{K}}^{\prime}=\left(  X_{K}\right)
_{\operatorname{mod}}^{\ast}$, $Y_{|X_{K}}^{\prime}=\left(  Y_{K}\right)
_{\operatorname{mod}}^{\ast}$ then $X$ is isomorphic with $Y$ in
$\mathcal{HF}_{\operatorname{mod}}(G)$.
\end{corollary}

\begin{proof}
We will use a theorem of Banach which we now recall. Assume that $Z$ and $W$
are Fr\'{e}chet spaces and let $u:Z\rightarrow W$ be a continuous linear map.
Set $u^{T}(\lambda)=\lambda\circ u$ for $\lambda\in W^{\prime}$ so
$u^{T}:W^{\prime}\rightarrow Z^{\prime}$. Then the theorem says that $u$ is
surjective if $u^{T}$ is injective and the image of $u^{T}$ is weakly closed
in $Z^{\prime}$ (see Treves [T,3.7.2]). We will identify $X_{K}$ and $Y_{K}$
and call them both $V$. Theorem \ref{left} implies that the identity map
$V\rightarrow V$ induces continuous morphisms $A:X\rightarrow\overline{V}$ and
$B:Y\rightarrow\overline{V}$. Let $Z=A(X)\cap B(Y)$ with topology given by the
seminorms for both $X$ and $Y$. Then $Z$ is complete, $G$--invariant and hence
an element of $\mathcal{HF}_{\operatorname{mod}}(G)$ and $Z_{K}=V$. There are
two morphisms $\alpha:Z\rightarrow X$ and $\beta:Z\rightarrow Y$ in
$\mathcal{HF}_{\operatorname{mod}}(G)$. Now $Z_{|V}^{\prime}\supset
X_{|V}^{\prime}=Y_{|V}^{\prime}=V_{\operatorname{mod}}^{\ast}\supset
Z_{|V}^{\prime}$ thus $\alpha^{T}$ and $\beta^{T}$ are bijective. Hence,
$\alpha$ and $\beta$ are surjective by Banach's theorem. This implies $X=Y$
(under our identification of $K$--finite vectors).
\end{proof}

We note that Lemma 11 combined with the above theorem implies that if
$V\in\mathcal{H(}\mathfrak{g},K)$ then there is up to isomorphism exactly one
$Z\in\mathcal{HF}_{\operatorname{mod}}(G)$ with $Z_{K}$ isomorphic with $V$
and $Z_{|Z_{K}}^{\prime}=(Z_{K})_{\operatorname{mod}}^{\ast}$. We will choose
one such $Z$ and denote it $\overline{\overline{V}}.$ Then $V\rightarrow
\overline{\overline{V}}$ defines a functor from $\mathcal{H(}\mathfrak{g},K)$
to $\mathcal{HF}_{\operatorname{mod}}(G)$. We will show that it has the
property described in Theorem \ref{minimal} after we prove Theorem 12.

\bigskip

The idea of the proof of Theorem \ref{crux} is to show that if $\lambda
\in\left(  H_{K}\right)  _{\operatorname{mod}}^{\ast}$ then there exists a
Hilbert representation of $G$, $(\pi_{1},H_{1})$ such that $\left(
H_{1}\right)  _{K}=\widehat{V}$ , so $(\widehat{\pi_{1}},H_{1})_{K}$ is
isomorphic with $V,$ under the isomorphism $T$ , and there exists an element
$u\in H_{1}$ so that $\left\langle Tv,u\right\rangle =\lambda(v)$ for $v$ in
$V$ here $\left\langle ...,...\right\rangle $ is the inner product on $H_{1}$.
In other words, there is a Hilbert completion of $V$ so that $\lambda$ is an
element of the conjugate dual Hilbert representation. We will give a complete
proof since we will use the technique in last section. We write $(...,...)$
for the inner product on $H$ which we assume is $K$--invariant. If $\mu\in
V_{\operatorname{mod}}^{\ast}$, $v\in V$ we will use the notation $f_{\mu,v}$
for the element $f$ of $\mathcal{A}_{\operatorname{mod}}(G)$ such that
$xf(k)=\mu(kxv)$ for $x\in U(\mathfrak{g})$ and $k\in K$.

Let $v_{1},...,v_{n}$ be an orthonormal basis of a $K$ and $Z(\mathfrak{g}%
)$--invariant subspace of $W$ in $V$ such that $V=U(\mathfrak{g})W$. Fix
$\lambda\in\left(  H_{K}\right)  _{\operatorname{mod}}^{\ast}$. Let $d$ and
$C$ be such that for every $g\in G$ we have

1. $\left\Vert \widehat{\pi}(g)\right\Vert \leq C\left\Vert g\right\Vert ^{d}$ and

2. $|f_{\lambda,v_{i}}(g)|\leq C\left\Vert g\right\Vert ^{d}$ for $i=1,...,n$.

We also choose $d_{o}$ such that%
\[
\int_{G}\left\Vert g\right\Vert ^{-d_{o}}dg<\infty.
\]

If $v,w\in H$ then we define a new inner product%
\[
\left\langle v,w\right\rangle =\sum_{i=1}^{n}\int_{G}(v_{i},\widehat{\pi
}(g)w)(\widehat{\pi}(g)v,v_{i})\left\Vert g\right\Vert ^{-2d-d_{o}}dg.
\]
We note that since%
\[
|(v_{i},\widehat{\pi}(g)v)(\widehat{\pi}(g)w,v_{i})|\leq\left\Vert
v_{i}\right\Vert ^{2}\left\Vert v\right\Vert \left\Vert w\right\Vert
\left\Vert \widehat{\pi}(g)\right\Vert ^{2}\leq C^{2}\left\Vert v_{i}%
\right\Vert ^{2}\left\Vert v\right\Vert \left\Vert w\right\Vert \left\Vert
g\right\Vert ^{2d}%
\]
the above integral converges for all $v,w\in H$ and defines a new inner
product on $H$ that is $K$--invariant (since $\left\Vert gk\right\Vert
=\left\Vert g\right\Vert $ for $g\in G$ and $k\in K$). Furthermore, if we set
$\left\Vert v\right\Vert _{1}^{2}=\left\langle v,v\right\rangle $ for $v\in H$
then%
\[
\left\Vert v\right\Vert _{1}\leq C_{1}\left\Vert v\right\Vert \qquad
\overset{}{(\ast)}%
\]
with $C_{1}=C\sqrt{n\int_{G}\left\Vert g\right\Vert ^{-d_{o}}dg}$. We set
$H_{1}$ equal to the Hilbert space completion of $H$ relative to $\left\Vert
...\right\Vert _{1}$. The above inequality implies that $H$ imbeds
continuously into $H_{1}$ via the canonical injection $(v\mapsto v).$ Noting
that $\widehat{V}$ is dense in $H_{1}$ we have $\left(  H_{1}\right)
_{K}=\widehat{V}$.

We next observe that if $v\in H$%
\[
\left\Vert \widehat{\pi}(x)v\right\Vert _{1}^{2}=\sum_{i=1}^{n}\int%
_{G}|(\widehat{\pi}(gx)v,v_{i})|^{2}\left\Vert g\right\Vert ^{-2d-d_{o}}dg=
\]%
\[
\sum_{i=1}^{n}\int_{G}|(\widehat{\pi}(g)v,v_{i})|^{2}\left\Vert gx^{-1}%
\right\Vert ^{-2d-d_{o}}dg\leq\left\Vert x\right\Vert ^{2d+d_{o}}\left\Vert
v\right\Vert _{1}%
\]
Here we use%
\[
\left\Vert g\right\Vert =\left\Vert gx^{-1}x\right\Vert \leq\left\Vert
gx^{-1}\right\Vert \left\Vert x\right\Vert
\]
so
\[
\left\Vert gx^{-1}\right\Vert ^{-2d-d_{o}}\leq\left\Vert g\right\Vert
^{-2d-d_{o}}\left\Vert x\right\Vert ^{2d+d_{o}}.
\]
This implies that $\widehat{\pi}(g)$ on $H$ extends to a strongly continuous
representation $(\pi_{1},H_{1})$ of $G$. \ Let $X\in Lie(G)$ and
$t\in\mathbb{R}^{\times}$. If $v$ is a $C^{\infty}$ vector in $H$ relative to
$\widehat{\pi}$ then using (*) \ above we have%
\[
\left\Vert \frac{\pi_{1}(\exp(tX))v-v}{t}-d\widehat{\pi}(X)v\right\Vert
_{1}=\left\Vert \frac{\widehat{\pi}(\exp(tX))v-v}{t}-d\widehat{\pi
}(X)v\right\Vert _{1}\leq
\]%
\[
C_{1}\left\Vert \frac{\widehat{\pi}(\exp(tX))v-v}{t}-d\widehat{\pi
}(X)v\right\Vert .
\]
Thus, $d\pi_{1}(X)v=d\widehat{\pi}(X)v$. So $(\pi_{1},H_{1})$ is a Hilbert
completion of $\widehat{V}$. Iterating this argument we see that if $Z\subset
H$ is the space of $C^{\infty}$ vectors relative to $\widehat{\pi}$ with the
usual Fr\'{e}chet topology then the imbedding of $Z$ into $H_{1}$ maps it
continuously into $H_{1}^{\infty}.$ On the other hand, we have $Z$ is
isomorphic in $\mathcal{HF}_{\operatorname{mod}}(G)$ with $\overline
{(\widehat{V})}$ so Theorem \ref{left} implies that the identity map on
$\widehat{V}$ induces a continuous isomorphism of $H_{1}^{\infty}$ into $Z$.
This implies that $Z=H_{1}^{\infty}$.

If $\gamma\in\widehat{K}$ let $E_{\gamma}$ denote the orthogonal projection to
the $K$--isotypic component ~$H_{1}(\gamma)=\widehat{V}(\gamma)$ we note that
$E_{\gamma|H}$ is also the orthogonal projection onto $H(\gamma)$. If $\mu\in
V^{\ast}$ then denote $\mu\circ E_{\gamma}$ by $\mu_{\gamma}$, then there is a
unique element $\tau(\mu_{\gamma})\in V(\gamma)$ such that
\[
\mu_{\gamma}(v)=(v,\tau(\mu_{\gamma})),v\in V.
\]

Note that $\tau$ is conjugate linear and we will identify $\mu_{\gamma}$ with
$\tau(\mu_{\gamma})$ as an element of $\widehat{V}(\gamma)$. Let for $\mu\in
V_{\operatorname{mod}}^{\ast}$ and $v\in V$, $f_{\mu,v}\in\mathcal{A}%
_{\operatorname{mod}}(G)$ be as in the beginning of this proof. We note that
if $\mu\in V(\gamma)^{\ast}$ then
\[
f_{\mu,v}(g)=(\pi(g)v,\tau(\mu))=(v,\widehat{\pi}(g)^{-1}\tau(\mu)).
\]
We now return to the element $\lambda$ in the beginning of the proof. We will
drop the $\tau$ and consider $\lambda_{\gamma}$ to be an element of
$\widehat{V}(\gamma)$. We now prove that $\sum\lambda_{\gamma}$ converges in
$H_{1}$. Observe that 2. above implies that%
\[
\int_{G}|f_{\lambda,v}(g)|^{2}\left\Vert g\right\Vert ^{-2d-d_{o}}dg<\infty.
\]
We note that if $\chi_{\gamma}$ is the character of $\gamma$ then
\[
f_{\lambda_{\gamma},v}(g)=d(\gamma)\int_{K}\chi_{\gamma}(k^{-1})f_{\lambda
,v}(kg)dk.
\]
Thus, using fact that $\left\Vert k_{1}gk_{2}\right\Vert =\left\Vert
g\right\Vert $ for all $k_{1},k_{2}\in K$ and $g\in G$, the Schur
orthogonality relations combined with the convergence of the $K$--Fourier
series in $L^{2}(G,\left\Vert g\right\Vert ^{-2d-d_{o}}dg)$ imply that%
\[
\int_{G}|f_{\lambda,v_{i}}(g)|^{2}\left\Vert g\right\Vert ^{-2d-d_{o}}%
dg=\sum_{\gamma\in\widehat{K}}\int_{G}|f_{\lambda_{\gamma},v_{i}}%
(g)|^{2}\left\Vert g\right\Vert ^{-2d-d_{o}}dg.
\]
Hence%
\[
\infty>\sum_{i=1}^{n}\sum_{\gamma\in\widehat{K}}\int_{G}|f_{\lambda_{\gamma
},v_{i}}(g)|^{2}\left\Vert g\right\Vert ^{-2d-d_{o}}dg=
\]%
\[
\sum_{i=1}^{n}\sum_{\gamma\in\widehat{K}}\int_{G}|(v_{i},\widehat{\pi}%
(g)^{-1}\lambda_{\gamma})|^{2}\left\Vert g\right\Vert ^{-2d-d_{o}}dg=
\]%
\[
\sum_{i=1}^{n}\sum_{\gamma\in\widehat{K}}\int_{G}|(v_{i},\widehat{\pi
}(g)\lambda_{\gamma})|^{2}\left\Vert g^{-1}\right\Vert ^{-2d-d_{o}}dg=
\]%
\[
\sum_{i=1}^{n}\sum_{\gamma\in\widehat{K}}\int_{G}|(v_{i},\widehat{\pi
}(g)\lambda_{\gamma})|^{2}\left\Vert g\right\Vert ^{-2d-d_{o}}dg=\sum
\left\Vert \lambda_{\gamma}\right\Vert _{1}^{2}.
\]

\begin{lemma}
\label{regularity}The topology of $H_{1}^{\infty}$ is given by the seminorms
\[
p_{l}(v)=\left\Vert (1+C_{K})^{l}v\right\Vert _{1}.
\]

\end{lemma}

We will give a proof in the appendix below, there is also a more elementary
proof in [RRGII] pp. 91-93. In the appendix we prove that this is also true
for any Hilbert representation for which the Casimir operator of $G$ acts by a
scalar on the $C^{\infty}$--vectors. We now continue with the proof of Theorem
\ref{crux}.

We note that $p_{l}(v)\leq p_{l+1}(v)$ for all $v\in H_{1}^{\infty}$. Since
$(\widehat{\pi},H^{\infty})=H_{1}^{\infty}$,\thinspace$\left\Vert
...\right\Vert $ defines a continuous norm on $H_{1}^{\infty}$. This implies
that there exists $l$ and $B>0$ such that%
\[
\left\Vert v\right\Vert \leq B\left\Vert (1+C_{K})^{l}v\right\Vert _{1}.
\]
Let $\mu(\gamma)I$ be the action of $C_{K}$ on $V_{\gamma}$. Then the above
inequality implies that
\[
(1+\mu(\gamma))^{-l}\left\Vert \lambda_{\gamma}\right\Vert \leq B\left\Vert
\lambda_{\gamma}\right\Vert _{1}\text{.}%
\]
This implies that
\[
\sum(1+\mu(\gamma))^{-l}\lambda_{\gamma}%
\]
converges in $H,$ to an element $u$. Thus if $v\in V$ then
\[
((I+C_{K})^{l}v,u)=\lambda(v).
\]
Hence
\[
\left\vert \lambda(v)\right\vert \leq B\left\Vert u\right\Vert \left\Vert
(1+C_{k})^{l}v\right\Vert .
\]
Thus $\lambda$ extends to a continuous functional on $(\pi,H^{\infty})$. This
completes the proof of Theorem \ref{crux}.

\bigskip

We now complete the proof of Theorem \ref{minimal}. Let $S:V\rightarrow W$ be
a surjective morphism in $\mathcal{H(}\mathfrak{g},K)$ then we have
$\widehat{S}:\widehat{W}\rightarrow\widehat{V}$ is an injective morphism. Let
$F=\widehat{V}/\mathfrak{n}^{k_{o}}\widehat{V}$ as a $P$ module then
$\overline{\widehat{V}}$ is the closure of $T_{k_{o}}(\widehat{V})$ in
$(\pi,H^{\infty})$ where $H$ is the closure of $T_{k_{o}}(\widehat{V})$ in the
Hilbert representation $Ind_{p}^{G}(F)$. For simplicity we replace
$\widehat{V}$ with $T_{k_{o}}(\widehat{V})$ and $\widehat{W}$ with it's image
in $\widehat{V}$. Let $H_{1}$ (not the same as the one in the proof above) be
the closure of $\widehat{W}$ in $H$ and if $\pi_{1}$ is the \ corresponding
\ action of $G$ on $H_{1}$ then $(\pi_{1},H_{1}^{\infty})$ is equivalent with
$\overline{\widehat{W}}$. Now $(\widehat{\pi},H)$ and $(\widehat{\pi_{1}%
},H_{1})$ are Hilbert realizations of $V$ and $W$ respectively. The
corresponding $C^{\infty}$ vectors are therefore respectively $\overline
{\overline{V}}$ and $\overline{\overline{W}}$ by Theorem \ref{crux}. If $v\in
H$ then $\lambda_{v}(w)=(v,w)$ defines an element of the conjugate dual. The
restriction of \thinspace$\lambda_{v}$, $v\in H$ to $H_{1}$ yields a Hilbert
representation surjection of $(\widehat{\pi},H)$ to $(\widehat{\pi_{1}}%
,H_{1})$ taking $C^{\infty}$ vectors completes the proof.

\subsubsection{Appendix: A proof of Lemma \ref{regularity}}

Recall that we have fixed a form $B$ on $Lie(G)$ given by $B(X,Y)=tr(XY)$. We
also note that $\left\langle X,Y\right\rangle =tr(XY^{T})$ defines an inner
product on $Lie(G)$. Let $C$ and $C_{K}$ be the Casimir operators of $G$ and
$K$ respectively corresponding to $B$ and we set $\Delta=C-2C_{K}$. We observe
that $\Delta=\sum X_{i}^{2}$ for $X_{1},...,X_{m}$ an orthonormal basis of
$Lie(G)$ relative to $\left\langle ...,...\right\rangle $. As a left invariant
operator on $G$, $\Delta$ is an elliptic and bi--invariant under $K$. Let
$(\pi,H)$ be a Hilbert representation of $G$ and set $V=\left(  H^{\infty
}\right)  _{K}$ . Let $Z$ be the completion of $V$ relative to the seminorms
$q_{l}(v)=\left\Vert \Delta^{l}v\right\Vert ,l=0,1,2,...$ Then since
$q_{0}=\left\Vert ...\right\Vert $, $Z$ can be looked upon as a subspace of
$H$. Also $H^{\infty}$ is the completion of $V$ using the seminorms
$s_{x}(v)=\left\Vert xv\right\Vert $ with $x\in U(\mathfrak{g})$. Thus
$Z\supset H^{\infty}$.

\begin{lemma}
\label{roe}$Z=H^{\infty}$. Furthermore, the topology on $H^{\infty}$ is given
by the semi-norms $q_{l}$.
\end{lemma}

\begin{proof}
We note that the second assertion is a direct consequence of the closed graph
theorem (c.f. [T]) and the first assertion. We will now prove the first
assertion. Let $v\in Z\subset H$. We must prove that $v\in H^{\infty}$. Let
$v_{j}\in V$ be a sequence converging to $v$ in the topology of $Z$. Let $w\in
H$ then for all $j$%
\[
\Delta^{k}(\pi(g)v_{j},w)=(\pi(g)\Delta^{k}v_{j},w).
\]
Furthermore, since $q_{l}(\Delta^{k}v)=q_{l+k}(v)$ and $q_{l}(v)\leq
q_{l+1}(v)$, we see that for fixed $k$ the sequence $\{\Delta^{k}v_{j}\}_{j}$
converges to $u_{k}$ in $Z$.

We assert that the function $g\mapsto(\pi(g)v,w)$ is $C^{\infty}$. Since $w\in
H$ is arbitrary this would imply that the map $g\mapsto\pi(g)v$ is weakly
$C^{\infty}$. But a weakly $C^{\infty}$ map of a finite dimensional manifold
into a Hilbert space is strongly $C^{\infty}$(c.f. [S]). This is exactly the
statement that $v$ is a $C^{\infty}$ vector. We now prove the assertion. We
first show that if we look upon the continuous function $h(g)=(\pi(g)v,w)$ as
a distribution on $G$ (using the Haar measure on $G)$ then in the distribution
sense
\[
\Delta^{k}h(g)=(\pi(g)u_{k},w).
\]
Indeed, let $f\in C_{c}^{\infty}(G)$ then%
\[
\int_{G}h(g)\Delta^{k}f(g)dg=\lim_{j\rightarrow\infty}\int_{G}(\pi
(g)v_{j},w)\Delta^{k}f(g)dg=
\]%
\[
\lim_{j\rightarrow\infty}\int_{G}\Delta^{k}(\pi(g)v_{j},w)f(g)dg=\lim
_{j\rightarrow\infty}\int_{G}(\pi(g)\Delta^{k}v_{j},w)f(g)dg=
\]%
\[
\int_{G}(\pi(g)u_{k},w)f(g)dg
\]
as asserted. Since $\Delta$ is elliptic, local Sobelev theory (c.f. [F,
Chapter 6]) implies that $h\in C^{\infty}(G)$.
\end{proof}

We will now assume that $V\in\mathcal{H}(\mathfrak{g},K)$ and give a proof of
Lemma \ref{regularity}. We note that
\[
\Delta^{k}=\sum_{j=0}^{k}(-2)^{j}\binom{k}{j}C^{k-j}C_{K}^{j}.
\]
If $v\in\widehat{V}$ then
\[
\left\Vert Cv\right\Vert _{1}^{2}=\sum_{i=1}^{n}\int_{G}|(v_{i},\widehat{\pi
}(g)d\widehat{\pi}(C)v)|^{2}\left\Vert g\right\Vert ^{-2d-d_{o}}=
\]%
\[
\sum_{i=1}^{n}\int_{G}|(d\pi(C)v_{i},\widehat{\pi}(g)v)|^{2}\left\Vert
g\right\Vert ^{-2d-d_{o}}.
\]
Now $d\pi(C)v_{i}=\sum a_{ji}v_{j}$ hence $(d\pi(C)v_{i},\widehat{\pi
}(g)v)=\sum a_{ji}(v_{j},\widehat{\pi}(g)v).$Hence setting $A=\max
_{ij}\{|a_{ij}|\}$
\[
|(d\pi(C)v_{i},\widehat{\pi}(g)v)|=|\sum_{j}a_{ji}(v_{j},\widehat{\pi
}(g)v)|\leq A\sum_{j}\left\vert (v_{j},\widehat{\pi}(g)v)\right\vert
\]
so%
\[
\sum_{i}(d\pi(C)v_{i},\widehat{\pi}(g)v)|^{2}\leq nA^{2}\left(  \sum
_{j}\left\vert (v_{j},\widehat{\pi}(g)v)\right\vert \right)  ^{2}\leq
\]%
\[
n^{2}A^{2}\sum_{i}|(v_{i},\widehat{\pi}(g)v)|^{2}.
\]
Thus%
\[
\left\Vert Cv\right\Vert _{1}\leq nA\left\Vert v\right\Vert _{1}.
\]
Set $B=nA.$Then%
\[
\left\Vert \Delta^{k}v\right\Vert _{1}=\left\Vert \sum_{j=0}^{k}(-2)^{j}%
\binom{k}{j}C^{k-j}C_{K}^{j}v\right\Vert _{1}\leq
\]%
\[
\sum_{j=0}^{k}(2)^{j}\binom{k}{j}B^{k-j}\left\Vert C_{K}^{j}v\right\Vert
_{1}\leq\sum_{j=0}^{k}(2)^{j}\binom{k}{j}B^{k-j}\left\Vert (1+C_{K}%
)^{j}v\right\Vert _{1}%
\]

This completes the proof of Lemma 14. 

We conclude this appendix with a result about general Hilbert representations.

\begin{proposition}
\label{C-K}If $(\pi,H)$ is a Hilbert representation of $G$ such that $C$ acts
as the scalar $c$ on $H^{\infty}$ then topology of $H^{\infty}$ is given by
the semi-norms $p_{l}(v)=\left\Vert (I+C_{K})^{l}v\right\Vert ,l=0,1,2,...$
\end{proposition}

\begin{proof}
Notice that in the proof of Lemma \ref{roe} the assumption of admissibility is
never used. Assume that $C$ acts by $c$ on $H^{\infty}$. So if $v\in
H^{\infty}$%
\[
\left\Vert \Delta^{k}v\right\Vert =\left\Vert \sum_{j=0}^{k}(-2)^{j}\binom
{k}{j}C^{k-j}C_{K}^{j}v\right\Vert _{{}}\leq
\]%
\[
\sum_{j=0}^{k}(2)^{j}\binom{k}{j}|c|^{\substack{\\k-j}}\left\Vert C_{K}%
^{j}v\right\Vert \leq\sum_{j=0}^{k}(2)^{j}\binom{k}{j}|c|^{k-j}\left\Vert
(1+C_{K})^{j}v\right\Vert .
\]

\end{proof}

We leave it to the reader to prove the analogous result with $C$ satisfying
$p(C)=0$ on $H^{\infty}$, $p(x)\in\mathbb{C}[x]$.

\subsection{Step 3: All objects in $\mathcal{H}(\mathfrak{g},K)$ are good.}

We first give some functorial properties of goodness (see for the proofs
[RRGII,11.7.2], [C,7.16]] Casselman uses the term regular for our good and
uses a single bar for our double bar completion and vice-versa). Recall that
good means $\overline{V}$ and $\overline{\overline{V}}$ are isomorphic as
elements of $\mathcal{HF}_{\operatorname{mod}}(G)$.

\begin{lemma}
\label{functorial}Let $V\in\mathcal{H}(\mathfrak{g},K)$

1. If $V$ is good then $\widehat{V}$ is good

2. If $V$ is good and if $W$ is a summand of $V$ then $W$ is good.

3. If $V\in\mathcal{H}(\mathfrak{g},K)$ and if every irreducible subquotient
of $V$ is good then $V$ is good.
\end{lemma}

We also have (see [RRGII,11.7.3])

\begin{lemma}
Let $Q=L_{Q}N_{Q}$ be a standard parabolic subgroup of $G$(i.e. $Q\supset P$)
with $L_{Q}=Q\cap\theta(Q)$ and $N_{Q}$ the unipotent radical of $Q$. Then if
$W$ is a good object in $\mathcal{H}(Lie(L_{Q}),K\cap L_{Q})$ and $N_{Q}$ acts
locally finitely then the corresponding induced $(\mathfrak{g},K)$--module is good.
\end{lemma}

This combined with the proof that square integrable representations are good (
[RRGII,11.7.4]) implies that tempered representations are good,

Finally we defer to Casselman's much cleaner version of the end game
([C,Section 9]).which involves the Langlands quotient theorem and a
deformation argument due to him (you can also see a repetition of the argument
in [RRGII] in [BK]). So

\begin{theorem}
If $V\in\mathcal{H}(\mathfrak{g},K)$ then $V$ is good.
\end{theorem}

\section{Adding parameters}

\subsection{Some results of van der Noort}

Let $\Omega$ be an open subset of $\mathbb{C}^{n}$ then following [vdN] we
will define a holomorphic family of objects in $\mathcal{H}(\mathfrak{g},K)$
in terms of the two conditions below. Before we can give the second condition
we will need to give a consequence of the first. Let $V$ be a $(Lie(K)\otimes
\mathbb{C},K)-$module and let $\pi:\Omega\times\mathfrak{g}\rightarrow End(V)$
be such that

1. For each $z\in\Omega$ the operators $\pi(z,X)$ for $X\in\mathfrak{g}$
define a representation of $\mathfrak{g}$ compatible with the $K$--action.

Thus we have a map $\pi:$ $\Omega\times U(\mathfrak{g)}\rightarrow End(V)$.

\begin{lemma}
If $W\subset U(\mathfrak{g})$ is a finite dimensional subspace and $v\in V$
then $\dim Span_{\mathbb{C}}\pi(\Omega,W)v<\infty$.
\end{lemma}

\begin{proof}
We may assume that $W$ is invariant under $Ad(K)$. Let $Z=span_{\mathbb{C}}Kv$
then $\dim Z<\infty$. Let $W\otimes Z=\bigoplus_{\gamma\in S}\left(  W\otimes
Z\right)  (\gamma)$ be its $K$--isotypic decomposition. Here $S$ is the subset
of $\widehat{K}$ consisting of those $\gamma$ such that $\left(  W\otimes
Z\right)  (\gamma)\neq\{0\}$. We note that $S$ is finite. Then 1. implies
that
\[
Span_{\mathbb{C}}\pi(\Omega,W)v\subset\bigoplus_{\gamma\in S}V(\gamma)
\]
which is finite dimensional.
\end{proof}

The second condition (holomorphy) is

2. For each $W\subset U(\mathfrak{g}),L\subset V$ respectively finite
dimensional subspaces the map $\Omega\times W\times L\rightarrow
span_{\mathbb{C}}\pi(\Omega,W)L$ given by $(z,w,u)\mapsto\pi(z,w)u$ is holomorphic.

A triple $(\Omega,\pi,V)$ will be called a holomorphic family of objects in
$\mathcal{H}(\mathfrak{g},K)$ if $\Omega$ is open in $\mathbb{C}^{n}$ for some
$n$, $V$ is an admissible $(Lie(K)\otimes\mathbb{C},K)$--module and $\pi$
satisfies 1. and 2.

\begin{example}
\label{induced}We first note that we can define the same concept for a
standard parabolic subgroup $Q$ of $G$ and $Q\cap K.$ Let $(\Omega,\sigma
,W)$be holomorphic family of $\mathcal{H}(Lie(Q)\otimes\mathbb{C},K\cap
Q)$--modules$.$ We form the $K$--finite induced representation $V=Ind_{M\cap
Q}^{K}(W)_{K}^{\infty}$ and let $\pi(z,X)$ be the action on $Ind_{Q}%
^{G}(\sigma(z,\cdot))_{K}^{\infty}$. Then $(\Omega,\pi,V)$ will be called a
parabolically induced holomorphic family. This includes the parabolically
induced representations $V_{Q,\sigma,\nu}$ with $\sigma$ admissible for
$M_{Q}$ and $\Omega=Lie(A_{Q})_{\mathbb{C}}^{\ast}$.
\end{example}

We will now state one of the main results in [vdN] which is a special case of
his Theorem 3.2.11.

\begin{theorem}
\label{types}Let $(\pi,V,\Omega)$ be a holomorphic family of objects in
$\mathcal{H}(\mathfrak{g},K)$ then for each $z_{o}\in\Omega$ there exists
$U\subset\Omega$ an open neighborhood of $z_{o}$ and $W\subset V$ a finite
dimensional subspace of $V$ such that $\pi(z,U(\mathfrak{g}))W=V$ for all
$z\in U$.
\end{theorem}

Fix a Cartan subalgebra, $\mathfrak{h}$, of $\mathfrak{g}$. Denote by
$Z(\mathfrak{g})$ the center of $U(\mathfrak{g})$. We use the Harish--Chandra
parametrization of homomorphisms of $Z(\mathfrak{g})$ to $\mathbb{C}$. That
is, $\chi=\chi_{\Lambda}$, with $\Lambda\in\mathfrak{h}^{\ast}$(determined up
to the action of the Weyl group). As part of his proof of the theorem above he
proves (Proposition 3.2.5 in [vdN])

\begin{theorem}
\label{inf}If $X$ is a compact subset of $\Omega$ then the set of
Harish--Chandra parameters of the generalized infinitesimal characters of the
$\pi(z,\cdot),$ $z\in X$  is contained in a compact subset of $\mathfrak{h}%
^{\ast}$.
\end{theorem}

\subsection{Families of Hilbert representations}

We now define a holomorphic family of admissible Hilbert representations of
$G$ to be a triple $(\Omega,\pi,H)$ with $\Omega$ an open subset of
$\mathbb{C}^{n}$ invariant under complex conjugation, $(\tau,H)$ a unitary
representation of $K$ such that the $K$--finite vectors form an admissible
$(Lie(K)\otimes\mathbb{C},K)$--module and $\pi$ a map from $\Omega\times G$ to
the bounded invertible operators on $H$ satisfying:

1. $\pi$ is a strongly continuous map of $\Omega\times G$ to the bounded
operators on $H$ such that $g\mapsto\pi(z,g)$ defines an admissible finitely
generated representation of $G$ for each $z\in\Omega.$

2. If $(...,...)$ is the inner product on $H$ then the map $z\mapsto
(\pi(z,g)v,w)$ is holomorphic for all $g\in G,v,w\in H$.

We define the conjugate dual family to be $\widehat{\pi}(z,g)=\pi
(z,g^{-1})^{\ast}$ we note that this family is antiholomorphic (i.e.
$\widehat{\pi}(\bar{z},g)$ defines a holomorphic family).

Let $C$ be the Casimir operator of $G$ chosen as as in section 2, For
simplicity, we will assume that on the $K$--finite vectors of $H$, $d\pi(z,C)$
acts by a scalar (depending on $z$). Thus if $V$ is the space of $K$--finite
vectors then the space of $C^{\infty}$ vectors for $\pi(z,\cdot)$ is equal to
the completion of $V$ relative to the seminorms $p_{l}(v)=\left\Vert
(I+C_{K})^{l}v\right\Vert ,l=1,2,...$ (see Proposition \ref{C-K}).\ 

\textbf{Example.} Let $Q$ be a standard parabolic subgroup of $G$. We can
define a holomorphic family of Hilbert representations of the pair $(Q,Q\cap
K)$ exactly as above. Let $(\Omega,\sigma,W)$ be a holomorphic family of
Hilbert representations. Thus $\sigma(z,.)_{|K\cap Q}$ is independent of $z$
that defines a unitary representation of $K\cap Q$ \ and the continuity and
holomorphic conditions are satisfied. Then $Ind_{Q}^{G}(\sigma(z,\cdot))$
defines a holomorphic family of Hilbert representations of $G$. This includes
the parabolically induced representations $I_{Q,\sigma,\nu}$ with $\sigma$ an
admissible representation of $M_{Q}$ on which the Casimir operator of $M_{Q}$
corresponding to $B$ (as above) acts by a scaler.

Let $(\Omega,\pi,H)$ be a holomorphic family of Hilbert representations of $G$
and let $d\pi(z,x)$ denote the corresponding action of $x\in U(\mathfrak{g})$
on $H^{\infty}$ and on $V$.

\begin{proposition}
$1$. $(\Omega,d\pi,V)$ is a holomorphic family of objects in $\mathcal{H}%
(\mathfrak{g},K)$.

$2$. If $\lambda\in\left(  H^{\infty}\right)  ^{\prime},v\in H^{\infty}$ then
the map $z,g\mapsto\lambda(\pi(z,g)v)$ is $C^{\infty}$ and holomorphic in
$\Omega$.
\end{proposition}

\begin{proof}
We leave the first assertion to the reader. The second assertion follows from
the fact that there exists $l\in\mathbb{Z}_{\geq0}$ and $u\in H$ such that
$\lambda(v)=((I+C_{K})^{l}v,u)$ for $v\in H^{\infty}$.
\end{proof}

We will say that a holomorphic family of Hilbert representations, $(\Omega
,\pi,H)$ is locally of uniform moderate growth if for each $z_{o}\in\Omega$
there exist $d_{z_{o}}$, $C_{z_{o}}$ and $U_{z_{o}}\subset\Omega$ an open
neighborhood of $z_{o}$ such that $\left\Vert \pi(z,g)\right\Vert \leq
C_{z_{o}}\left\Vert g\right\Vert ^{d_{z_{o}}}$ for $z\in U_{z_{o}},g\in G$.
One can check that a parabolically induced family satisfies this condition.

\subsection{Theorem \ref{characterization} with parameters}

The CW theorem can be restated in the following form: If $V\in\mathcal{HF}%
_{\operatorname{mod}}(G)$ then $V$ is isomorphic with $\overline{\left(
V_{K}\right)  }$ which is isomorphic with $\overline{\overline{\left(
V_{K}\right)  }}$. The theorem therefore says

\begin{theorem}
Let $V\in\mathcal{HF}_{\operatorname{mod}}(G)$ then $\overline{V}_{|V_{K}%
}^{\prime}=\left(  V_{K}\right)  _{\operatorname{mod}}^{\ast}$ (see subsection 4.3).
\end{theorem}

Let $(\Omega,\sigma,V)$ be a holomorphic family of objects in $\mathcal{H}%
(\mathfrak{g},K)$. We define $V_{z}\in$ $\mathcal{HF}_{\operatorname{mod}}(G)$
to be $(\sigma(z,\cdot),V)$ for $z\in\Omega$. \ If $\lambda_{z}\in
(V_{z})_{\operatorname{mod}}^{\ast}$ for each $z\in\Omega$ then $z\longmapsto
\lambda_{z}$ will be called holomorphic if the correspondence $z\mapsto
\lambda_{z}(v)$ is holomorphic for all $v\in V$. We will say that this family
is locally of uniform moderate growth if for each $v\in V$ and $z_{o}\in
\Omega$ there is $U_{z_{o}}\subset\Omega$ an open neighborhood of $z_{o}$ and
$C_{U_{z_{o}},v}$ and $d_{U_{z_{o}},v}$ such that (in the notations of
subsection 4.3)%
\[
|f_{\lambda_{z},v}(g)|\leq C_{U_{z_{o}},v}\left\Vert g\right\Vert
^{d_{U_{z_{o}},v}},z\in U.
\]
Our main result is

\begin{theorem}
Assume that $(\Omega,\sigma,V)$ is a Holomorphic family of objects in
$\mathcal{H}(\mathfrak{g},K)$ such that there is a holomorphic family of
admissible Hilbert representations $(\Omega,\pi,H)$ of local uniform moderate
growth such that $\sigma=d\pi$ and $V=H_{|K}$. If $z\mapsto\lambda_{z}%
\in(V_{z})_{\operatorname{mod}}^{\ast}$ is a holomorphic on $\Omega$ of local
uniform moderate growth \ on $\Omega$ and if we (also) denote the extension of
$\lambda_{z}$ to $H^{\infty}$ by $\lambda_{z}$ then $z\mapsto\lambda_{z}$ is
weakly holomorphic from $\Omega$ to $\left(  H^{\infty}\right)  ^{\prime}$.
\end{theorem}

\begin{proof}
We will follow the proof of Theorem \ref{crux} which can be found in
subsection \ref{proofmin}. We will use notation from that subsection. Let
$z_{o}\in\Omega$ and let $U\subset\Omega$ be an open neighborhood of $z_{o}$
such that there exists $d$ and $C>0$ so that the following 3 conditions are
satisfied for all $z\in U$

1. $\left\Vert \widehat{\pi}(z,g)\right\Vert \leq C\left\Vert g\right\Vert
^{d}$ for some $C>0$ and

2. $v_{1},...,v_{n}$ an orthonormal basis of a finite sum of $K$--isotypic
components $W\subset V$ such that $V=d\pi(z,U(\mathfrak{g}))W$ for $z\in U$.

3. $|f_{\lambda_{z},v_{i}}(g)|\leq C\left\Vert g\right\Vert ^{d}$ for
$i=1,...,n$.

We note that 2. is possible by Theorem \ref{types} above. 3. is a consequence
of the local uniformity.

If $z\in U,v,w\in H$ we define ($d_{o}$ is as in section \ref{proofmin})
\[
\left\langle v,w\right\rangle _{z}=\sum_{i=1}^{n}\int_{G}(v_{i},\hat{\pi
}(z,g)v)\overline{(v_{i},\hat{\pi}(z,g)w)}\left\Vert g\right\Vert ^{-2d-d_{o}%
}dg.
\]

We observe that if we argue as in Theorem \ref{crux} and $\left\Vert
v\right\Vert _{z}$ \ is the corresponding norm we have for all $z\in U,v\in H$%
\[
\left\Vert v\right\Vert _{z}\leq C_{1}\left\Vert v\right\Vert \qquad
\overset{}{(\ast)}%
\]

with $C_{1}=C\sqrt{n\int_{G}\left\Vert g\right\Vert ^{-d_{o}}dg}$. The
inequality (*) implies that $\left(  H_{z}^{\infty}\right)  _{K}$ is
isomorphic with $\widehat{V}_{z}$. \ As in the proof of the theorem without
parameters if $\lambda_{z,\gamma}=\lambda_{z}\circ E_{\gamma}$ then we have%
\[
\sum_{\gamma\in\widehat{K}}\left\Vert \lambda_{z,\gamma}\right\Vert _{z}%
^{2}\leq C_{2}<\infty\qquad\overset{}{(\ast\ast)}%
\]
with $C_{2}=nC^{2}\int_{G}\left\Vert g\right\Vert ^{-d_{o}}dg$. Thus
$\lambda_{z}$ defines an element of $H_{z}$. We now note that $\left\langle
v,w\right\rangle _{z}=(A_{z}v,w)$ for $v,w\in H$ and $A_{z}$ is a bounded self
adjoint positive operator on $H$. We also note that $z\mapsto A_{z}$ is real
analytic in the strong operator topology (this is because weak analyticity is
the same as strong analyticity). Here we use
\[
\left\langle v,w\right\rangle _{z}=\sum_{i=1}^{n}\int_{G}(\pi(z,g)v_{i}%
,v)\overline{(\pi(z,g)v_{i},w)}\left\Vert g\right\Vert ^{-2d-d_{o}}dg.
\]

Also, since the topology of \ $(\widehat{\pi(z)},H_{z}^{\infty})$ is given by
the seminorms $v\mapsto\left\Vert (I+C_{K})^{l}v\right\Vert _{z}=\left\Vert
(I+C_{K})^{l}A_{z}^{\frac{1}{2}}v\right\Vert $ or by the semi-norms
$v\mapsto\left\Vert (I+C_{K})^{l}v\right\Vert $ this implies that $B_{z}%
=A_{z}^{\frac{1}{2}}$ defines an isomorphism of $H_{z}^{\infty}$. The upshot
is
\[
\sum B_{z}\lambda_{z,\gamma}=u_{z}\in H.
\]
The map $z\mapsto u_{z}$ is real analytic. We therefore have $\lambda
_{z}=B_{z}^{-1}u_{z}$ so it is real analytic from \thinspace$U$ into $\left(
H^{\infty}\right)  ^{\prime}.$ Since $z\mapsto\lambda_{z}(v)$ is holomorphic
for $v$ in the dense set $V$ we see that $z\mapsto\lambda_{z}$ is weakly
holomorphic from $U$ into $\left(  H^{\infty}\right)  ^{\prime}$.
\end{proof}

\subsubsection{Holomorphic families of automorphic forms and Eisenstein
series.}

We assume that $G$ has compact center and we consider the example of
automorphic forms in subsection \ref{automorphic}. Let $\Gamma$ be subgroup of
$G$ of finite covolume and let $\Omega$ be open in $\mathbb{C}^{m}$. Let
$f:\Omega\times G\rightarrow\mathbb{C}$ be real analytic such that $f(z,g)$,
is holomorphic in $z$ and automorphic on $G$ relative to $\Gamma$ then the
theory of the constant term implies the local uniformity of the previous theorem.

We note that the next result will use the work in Langlands [L] so all one
needs is to assume that $G$ and $\Gamma$ satisfy the hypotheses in the first
chapter of [L]. \ This material is difficult and (amazingly) the fact that his
properties are actually satisfied is not proved until the end of his (the
notorious) Chapter 7. We will therefore work in the simpler context of
arithmetic groups. That is we assume that $G$ is the set of real points of an
algebraic group defined over $\mathbb{Q}$ and $\Gamma$ is arithmetic with
respect to this $\mathbb{Q}$--structure. Let $Q$ be the real points of a
parabolic subgroup of $G$ defined over $\mathbb{Q}$. We choose $K$ so that $Q$
is $K$--standard. Before we can state the application to Eisenstein series we
need some notation. Let $U$ be the unipotent radical of $Q$ and let
$Q=LA_{Q}U$ be a standard $\mathbb{Q}$--Langlands decomposition of $Q$ (i.e.
$A_{Q}$ is the set of real points of a $\mathbb{Q}$--split torus that projects
onto the real points of a maximal $\mathbb{Q}$--split torus of $Q/U$). We
identify $L$ with $Q/A_{Q}U$. Set $\Gamma_{L}=\left(  \Gamma\cap Q\right)
/\left(  \Gamma\cap(A_{Q}U)\right)  $ and define $I(\Gamma_{L})$ to be the
space of all $f\in C^{\infty}(\Gamma_{L}\backslash L\times K)$ such that

1. If $x\in L,u\in K\cap L,k\in K$ then $f(xu,u^{-1}k)=f(x,k)$.

2. With $Z(Lie(L))$ (the center of $U(Lie(L)$) acting on the first factor we
have $\dim Z(Lie(L))f<\infty.$

3. With $U(Lie(L))$ acting on the first factor and $U(Lie(K))$ acting on the
second here exist $C_{f,x,y}$ and $d_{f}$ such that%
\[
|xyf(u,k)|\leq C_{f,x,y}|u|^{d_{f}},x\in U(Lie(L)),y\in U(Lie(K)),k\in K,u\in
L.
\]

4. For each $k\in K$ the map $x\longmapsto f(x,k)$ is a $\Gamma_{L}$--cusp
form. That is, if $P$ is the group of real points of parabolic subgroup of $L$
defined over $\mathbb{Q}$ with unipotent radical $U_{P}$ then
\[
\int_{U_{P}\cap\Gamma_{L}\backslash U_{P}}f(ux,k)du=0
\]
for all $x\in L$ and $k\in K$.

Let $\mu:Z(Lie(L))\rightarrow\mathbb{C}$ be a homomorphism and let $I_{\mu
}(\Gamma_{L})$ be the subspace of $I(\Gamma_{L})$ consisting of elements $h$
such that $uh=\mu(u)h$. We define an inner product on $I_{\mu}(\Gamma_{L})$:%
\[
(h_{1},h_{2})=\int_{\Gamma_{L}\backslash L\times K}h_{1}(u,k)\overline
{h_{2}(u,k)}dudk.
\]
We have an admissible unitary representation of $K$ on the Hilbert space
completion $H_{\mu}(\Gamma_{L})$ of $I_{\mu}(\Gamma_{L})$ by right translation
in the second variable. For $\nu\in Lie(A_{Q})_{\mathbb{C}}^{\ast}$ and $f\in
I_{\mu}(\Gamma_{L})$ \ we define $\widetilde{f}_{\nu}(uaxk)=a^{\nu+\rho_{_{Q}%
}}f(x,k)$ for $u\in U,a\in A_{Q},x\in L$ and $k\in K$. \ Then $\widetilde{f}%
_{\nu}\in C^{\infty}(G)$. We define $\pi(\nu,x)f(u,k)=R(x)\widetilde{f}_{\nu
}(uk)$ (here $R(x)\phi(g)=\phi(gx)$). This defines a holomorphic family of
Hilbert representations on $H_{\mu,\nu}(\Gamma_{L})$ of local (in $\nu$)
uniform moderate growth.

We note that $\left(  A_{Q}U\cap\Gamma\right)  /\left(  U\cap\Gamma\right)  $
is finite. If $f\in I_{\mu}(\Gamma_{L})$ we define%
\[
f_{\nu}(g)=\sum_{\gamma\in\left(  A_{Q}U\cap\Gamma\right)  /\left(
U\cap\Gamma\right)  }\widetilde{f_{\nu}(}\gamma g)\text{.}%
\]
Finally, the corresponding Eisenstein series is%
\[
E(Q,f,\nu)(g)=\sum_{\gamma\in\Gamma\cap U\backslash\Gamma}f_{\nu}(\gamma g).
\]
Langlands has proven a meromorphic continuation of $K$--finite Eisenstein
series ([L]) (that is $f$ is right $K$--finite in the $K$ variable). We can
think of these series as giving a meromorphic family, $\lambda_{\nu}$, of
elements of $((\pi(\nu),H_{\mu}(\Gamma_{L}))_{K})_{\operatorname{mod}}^{\ast}%
$. For each such and each $\nu_{o}\in Lie(A_{Q})_{\mathbb{C}}^{"}$ there
exists an open neighborhood of $\nu_{o}$, $\Omega_{1}$, and a holomorphic
function, $h$ on $\Omega_{1}$ such that $h(\nu)E(Q,f,\nu)$ is holomorphic on
$\Omega_{1}$. Let $\nu_{o}\in\Omega\subset\overline{\Omega}\subset\Omega_{1}$
with $\Omega$ open with compact closure. Then on $\Omega$ the family is of
local moderate growth. So Theorem 26 implies

\begin{theorem}
If $f\in I(\Gamma_{L})$ then the Eisenstein series $E(Q,f,\nu)$ initially
defined and holomorphic for $\operatorname{Re}\nu(\check{\alpha})>C\rho
_{Q}(\check{\alpha})$ for some $C$ depending on $f$ and $Q$ and all coroots of
roots appearing in the unipotent radical of $Q$ has a meromorphic continuation
to $(Lie(A)\otimes\mathbb{C})^{\ast}$.
\end{theorem}

\begin{center}
{\Large References}

\medskip
\end{center}

\noindent\lbrack BK] J. Bernstein and B. Kr\"{o}tz, Smooth Fr\'{e}chet
globalizations of Harish-Chandra modules, to appear\medskip

\noindent\lbrack BW] A. Borel and N. Wallach, Continuous cohomology, discrete
subgroups and representations of reductive groups, Second Edition,
Mathematical Surveys and Monographs,Volume 67, AMS, Providence, RI,
2000.\medskip

\noindent\lbrack C] W. Casselman, Canonical extensions of Harish-Chandra
modules to representations of $G,$Canadian J, Math. 41(1989),286-320.\medskip

\noindent\lbrack F] Gerald B. Folland, Introduction to partial differential
equations, second edition, Princeton University Press, Princeton, NJ,
1995.\medskip

\noindent\lbrack K] B. Kostant,On the existence and irreducibility of certain
series of representations, Bull. Amer. Math. Soc. , 75(1969), 627-642.\medskip

\noindent\lbrack L] Robert P. Langlands, On the functional equations satisfied
by Eisenstein series, Lecture Notes in Mathematics 544, Springer-Verlag,
Berlin, 1974.

\noindent\lbrack Le] Seung Lee, Representations with small $K$--types, to
appear.\medskip

\noindent\lbrack OS] T. Oshima and J. Sekiguchi, Eigenspaces of invariant
differential operators on an affine symmetric space, Inventiones Math.,
57(1980).\medskip

\noindent\lbrack Sc] G. Schiffman, Integrals d'interlacement et fonctions de
Whittaker, Bull. Soc. Math. France, 99(1971), 3-27.\medskip

\noindent\lbrack S] Laurant Schwartz, Espaces de fonctiones
diff\'{e}rentielles a valeurs vectorielles, J. Analyse Math. (1954/55),
78-148.\medskip

\noindent\lbrack T] Fran\c{c}ois Treves, Topological vector spaces
distributions and kernels, Academic Press, New York, 1967.\medskip

\noindent\lbrack VW] D. A. Vogan and N.R. Wallach, Intertwining operators for
real reductive groups, Ad. in Math.,82(1990),203-243.\medskip

\noindent\lbrack vdN] Vincent van der Noort, Analytic parameter dependence of
Harish-Chandra modules for real reductive groups, A family affair, Thesis
University of Utrecht, 2009.\medskip

\noindent\lbrack RRGI,II] Nolan R. Wallach, Real reductive groups,I,II,
Academic Press, San Diego, 1988,1992.\medskip

\noindent\lbrack W] Nolan R. Wallach, Basic geometric invariant theory, to appear.

\end{document}